\documentclass[a4paper,reqno]{amsart}
\usepackage{amssymb}
\usepackage[cp1251]{inputenc}
\usepackage[english,russian]{babel}

\makeatletter
\def\@@and{и}
\makeatother

\PII{}
\makeatletter
\def\@setcopyright{\@empty}
\makeatother

\newcommand{\E}{E_n(f)_{p,\alpha}}
\newcommand{\Epar}[2]{E_{#1}\left(#2\right)_{p,\alpha}}
\newcommand{\T}[3]{\tau_{#1}^{#2}\left(#3,\mu\right)}
\newcommand{\hatT}[3]{\hat\tau_{#1}^{#2}\left(#3,\mu\right)}
\newcommand{\w}{\hat\omega(f,\delta,\mu)_{p,\alpha}}
\newcommand{\wpar}[1]
	{\hat\omega\left(#1,\mu\right)_{p,\alpha}}
\newcommand{\K}{K(f,\delta,\mu)_{p,\alpha}}
\newcommand{\Kpar}[1]{K\left(#1,\mu\right)_{p,\alpha}}
\newcommand{\norm}[1]{\left\|#1\right\|_{p,\alpha}}
\newcommand{\Si}[1]{(1-#1^2)}
\newcommand{\Co}[1]{(\cos #1/2)^{2\mu}}
\newcommand{\cosmu}{\cos\mu(\varphi_1-\varphi)}
\newcommand{\Dx}{D_{x,\mu,\mu}}
\newcommand{\Dy}{D_{y,0,2\mu}}
\newcommand{\AD}{AD(p,\alpha,\mu)}
\newcommand{\Px}[1]{P_{#1}^{(\mu,\mu)}}
\newcommand{\Py}{P_n^{(0,2\mu)}(y)}
\newcommand{\kap}{\varkappa(\delta)}
\newcommand{\sincosv}{(\sin v/2)^{-1}(\cos v/2)^{-4\mu-1}}
\newcommand{\sincosu}{(\sin u/2)(\cos u/2)^{4\mu+1}}
\newcommand{\varsincosu}{(\sin u/2)(\cos u/2)^{2\mu+1}}
\newcommand{\krn}{%
  \left(\frac{\sin\frac{mt}2}{\sin\frac t2}\right)^{2q+4}
	(\sin t)^{2\mu+1}}
\newcommand{\Lp}{L_{p,\alpha}}
\newcommand{\Lmu}{L_{1,\mu}}
\newcommand{\gd}{g_\delta(x)}
\newcommand{\allp}{1\le p\le\infty}
\newcommand{\allmu}{\mu\in\numN\cup\{0\}}

\newcommand{\numericset}[1]{\mathbb #1}
\newcommand{\numN}{\numericset N}

\newtheorem{thm}{Теорема}[section]
\newtheorem{lmm}{Лемма}[section]
\newtheorem*{cor}{Следствие}

\newcounter{const}
\numberwithin{const}{thm}
\numberwithin{const}{lmm}

\makeatletter
\newcommand{\Cn}[1][]{%
  \stepcounter{const}C_{\theconst}%
  \@ifnotempty{#1}{\newcounter{#1}\setcounter{#1}{\arabic{const}}}}
\makeatother

\newcommand{\lastC}{C_{\theconst}}
\newcommand{\prevC}[1][1]{%
	{\countdef\n=255
	 \n=\theconst
	 \advance\n by-#1
	 C_{\number\n}}}

\renewcommand{\theconst}{\arabic{const}}

\begin{document}

\title[О теореме Джексона\dots]{О теореме Джексона для
	модуля гладкости, определяемого несимметричным оператором
	обобщенного сдвига}
\author{M.~K.\ Потапов}
\address{M.~K.\ Потапов\\
	Механико-математический факультет\\
	Московский Государственный Университет им.\ Ломоносова\\
	Москва 117234\\
	Россия}
\author{Ф.~M.\ Бериша}
\address{Ф.~M.\ Бериша\\
	Механико-математический факультет\\
	Московский Государственный Университет им.\ Ломоносова\\
	Москва 117234\\
	Россия}
\curraddr{F.~M.\ Berisha\\
	Faculty of Mathematics and Sciences\\
	University of Prishtina\\
	N\"ena Terez\"e~5\\
	10000 Prishtina\\
	Kosovo}
	\email{faton.berisha@uni-pr.edu}

\thanks{Работа выполнена при поддержке Российского Фонда
	Фундаментальных Исследования (грант No.\ 96--01--00094)
	и программи поддержки ведущих научних школ
	(грант No.\ 96/97--15--96073).}

\date{}

\begin{abstract}
	В этой работе
	рассматривается класс несимметричных операторов обобщенного сдвига,
	для каждого из них
	вводится обобщенные модули гладкости
	и для них доказывается теорема Джексона
	и теорема, обратная ей.
\end{abstract}

\maketitle

\begin{otherlanguage}{english}
	\begin{abstract}
		In this paper
		a class of asymmetrical operators of generalised translation
		is introduced,
		for each of them
		generalised moduli of smoothness are introduced,
		and Jackson's and its converse theorems
		are proved for those moduli.
	\end{abstract}
\end{otherlanguage}

\section{Введение}

В работе~\cite{potapov:vestnik-98}
был введен в рассмотрение
несимметричный оператор обобщенного сдвига,
с его помощью был определен обобщенный модуль гладкости
и для него доказана теорема о совпадении класса функций,
определяемого этим модулем,
с классом функций, имеющих данной порядок наилучшего приближения
алгебраическими многочленами.

В этой работе
рассматривается класс несимметричных операторов обобщенного сдвига,
для каждого из них вводится обобщенные модули гладкости
и для них доказывается теорема Джексона
и теорема, обратная ей.

\section{Определение обобщенного модуля гладкости}

Обозначим через~$L_p$, $1\le p<\infty$,
множество функций~$f$,
измеримых по Лебегу и суммируемых в $p$-й степень
на отрезке $[-1,1]$,
а через~$L_\infty$ обозначим множество функций,
непрерывных на отрезке $[-1,1]$,
причем
\begin{displaymath}
	\|f\|_p=
	\begin{cases}
		\bigl(\int_{-1}^1|f(x)|^p\,dx\bigr)^{1/p},
		  &\text{если $1\le p<\infty$},\\
		\max_{-1\le x\le1}|f(x)|,
		  &\text{если $p=\infty$}.
	\end{cases}
\end{displaymath}

Через~$\Lp$ обозначим множество функций~$f$,
таких, что $f(x)\*\Si{x}^\alpha\in L_p$,
причем
\begin{displaymath}
	\norm f=\|f(x)\Si{x}^\alpha\|_p.
\end{displaymath}

Через~$\E$ обозначим наилучшее приближение функций~$f$
при помощи алгебраических многочленов степени не выше, чем $n-1$,
в метрике~$\Lp$,
т.е.
\begin{displaymath}
	\E=\inf_{P_n}\norm{f-P_n},
\end{displaymath}
где~$P_n$ --- алгебраический многочлен степени не выше,
чем $n-1$.

Пусть $\allmu$.
Для суммируемой функции~$f$
введем оператор обобщенного сдвига по правилу
\begin{multline*}
	\hatT t{}{f,x}=\frac1{\pi\Si{x}^{\mu/2}\Co t}\\
	\times\int_0^\pi\Si{R}^{\mu/2}f(R)\cosmu\,d\varphi_1,
\end{multline*}
где
\begin{equation}\label{eq:R-B}
	\begin{gathered}
	x=\cos\theta_1, \quad y=\cos t, \quad z=\cos\varphi_1,\\
	R=xy-z\sqrt{1-x^2}\sqrt{1-y^2}=\cos\theta,\\
	\sin\theta\cos\varphi=\cos\theta_1\sin t
		+\sin\theta_1\cos t\cos\varphi_1,\\
	\sin\theta\sin\varphi=\sin\theta_1\sin\varphi_1.
	\end{gathered}
\end{equation}

При помощи этого оператора обобщенного сдвига
определим обобщенный модуль гладкости
по правилу
\begin{displaymath}
	\w=\sup_{|t|\le\delta}\norm{\hatT t{}{f,x}-f(x)}.
\end{displaymath}

Полагая $y=\cos t$, $z=\cos\varphi_1$
в операторе $\hatT t{}{f,x}$,
обозначим его через $\T y{}{f,x}$
и перепишем в виде
\begin{multline*}
	\T y{}{f,x}=\frac{2^\mu}{\pi\Si{x}^{\mu/2}(1+y)^\mu}\\
		\times\int_{-1}^1\Si{R}^{\mu/2}f(R)\cosmu\frac{dz}{\sqrt{1-z^2}},
\end{multline*}
где~$R$, $\varphi_1$ и~$\varphi$
определены формулами~\eqref{eq:R-B}.

Обозначим через~$D_{x,\nu,\mu}$ операторы дифференцирования
определяемые по правилу
\begin{displaymath}
	D_{x,\nu,\mu}=\Si{x}\frac{d^2}{dx^2}
		+(\mu-\nu-(\nu+\mu+2)x)\frac d{dx}.
\end{displaymath}
Очевидно, что
\begin{displaymath}
	D_{x,\nu,\mu}=(1-x)^{-\nu}(1+x)^{-\mu}\frac d{dx}
	(1-x)^{\nu+1}(1+x)^{\mu+1}\frac d{dx}.
\end{displaymath}

Будем писать, что $f(x)\in\AD$,
если $f\in\Lp$,
$f(x)$ имеет абсолютно непрерывную производную
на каждом отрезке $[a,b]\subset(-1,1)$
и $\Dx f(x)\in\Lp$.

Пусть
\begin{displaymath}
	\K=\inf_{g\in\AD}\bigl(\norm{f-g}+\delta^2\norm{\Dx g(x)}\bigr)
\end{displaymath}
$K$--функционал типа Петре,
интерполирующий между пространствами~$\Lp$ и $\AD$.

Будем обозначать через~$P_n^{(\nu,\mu)}(x)$ $(n=0,1,\dotsc)$
многочлены Якоби,
т.е. многочлены степени~$n$,
ортогональные друг другу с весом $(1-x)^\nu\*(1+x)^\mu$
на отрезке $[-1,1]$
и нормированные условием $P_n^{(\nu,\mu)}(1)=1$ $(n=0,1,\dotsc)$.

Пусть $\allmu$.
Через~$a_n(f)$ обозначим коэффициенты Фурье--Якоби
функции $f\in\Lmu$
по системе многочленов Якоби
$\left\{\Px n(x)\right\}_{n=0}^\infty$,
т.е.
\begin{displaymath}
	a_n(f)=\int_{-1}^1f(x)\Px n(x)\Si{x}^\mu\,dx
	\quad
	(n=0,1,\dots).
\end{displaymath}

Определим следующие симметричные операторы обобщенного сдвига,
играющие в дальнейшем вспомогательную роль
\begin{displaymath}
	T_y(f,x,\mu)
	=\frac1{\gamma_\mu}\int_{-1}^1\Si{z}^{\mu-1/2}f(R)\,dz,
\end{displaymath}
где
\begin{displaymath}
	\gamma_\mu=\int_{-1}^1\Si{z}^{\mu-1/2}\,dz,
\end{displaymath}
$\allmu$, $R$ --- определено формулами~\eqref{eq:R-B}.

Целью этой статьи является доказательство следующего утверждения.

\begin{thm}\label{th:jackson}
	Пусть даны числа~$p$, $\mu$ и~$\alpha$ такие,
	что $\allp$, $\allmu$;
	\begin{alignat*}2
		-\frac12 	&<\alpha-\frac\mu2\le0
			&\quad &\text{при $p=1$},\\
		-\frac1{2p} &<\alpha-\frac\mu2<\frac12-\frac1{2p}
			&\quad &\text{при $1<p<\infty$},\\
		0 			&\le\alpha-\frac\mu2<\frac12
			&\quad &\text{при $p=\infty$}.
	\end{alignat*}
	Пусть $f\in\Lp$.
	Тогда для любого натурального~$n$
	справедливы неравенства
	\begin{displaymath}
		\Cn\E\le\wpar{f,\frac1n}
		\le\Cn\frac1{n^2}\sum_{\nu=1}^n\nu\Epar\nu f,
	\end{displaymath}
	где положительные постоянные~$\prevC$ и~$\lastC$
	не зависят от~$f$ и~$n$.
\end{thm}

Заметим,
что при $\mu=0$
теорема доказана в работе~\cite{potapov:vestnik-83}.
При $\mu=2$
теорема доказана в работе~\cite{p-berisha:anal-99}.

\section{Вспомогательные утверждения}

\begin{lmm}\label{lm:properties-tau}
	Оператор~$\T y{}{f,x}$ обладает следующими свойствами:
	\begin{enumerate}
	\item\label{it:properties-tau-1}
		Оператор~$\T y{}{f,x}$
		линеен по~$f$;
	\item\label{it:properties-tau-2}
		$\T1{}{f,x}=f(x)$;
	\item\label{it:properties-tau-3}
		$\T y{}{\Px n,x}=\Px n(x)\Py
		\quad
		(n=0,1,\dotsc)$;
	\item\label{it:properties-tau-4}
		$\T y{}{1,x}=1$;
	\item\label{it:properties-tau-5}
		Если $g(x)\T y{}{f,x}\in\Lmu$ для любого~$y$,
		то
		\begin{displaymath}
			\int_{-1}^1f(x)\T y{}{g,x}\Si{x}^\mu\,dx
			=\int_{-1}^1g(x)\T y{}{f,x}\Si{x}^\mu\,dx;
		\end{displaymath}
	\item\label{it:properties-tau-6}
		$a_n(\T y{}{f,x})=a_n(f)\Py
		\quad
		(n=0,1,\dotsc)$.
	\end{enumerate}
\end{lmm}

\begin{proof}
	Свойства~\ref{it:properties-tau-1}
	и~\ref{it:properties-tau-2}
	следуют сразу из определения оператора $\T y{}{f,x}$.
	
	Для доказательства свойства~\ref{it:properties-tau-3}
	рассмотрим функции
	\begin{displaymath}
	P_{\mu\nu}^l(z)=P_n^{(\alpha,\beta)}(z)2^{-\mu}i^{\mu-\nu}
		\sqrt{\frac{(l-\mu)!\,(l+\mu)!}{(l-\nu)!\,(l+\nu)!}}
		(1-z)^{\frac{\mu-\nu}2}(1+z)^{\frac{\mu+\nu}2},
	\end{displaymath}
	где $l=n+\frac{\alpha+\beta}2$,
	$\mu=\frac{\alpha+\beta}2$,
	$\nu=\frac{\beta-\alpha}2$.
	Полагая $k=\mu$, $\nu=0$
	в формуле умножения~\cite[с.140]{vilenkin:spetsiyalnye}
	для функции~$P_{\mu\nu}^l(z)$,
	получим равенство
	из свойства~\ref{it:properties-tau-3}.
	
	Свойство~\ref{it:properties-tau-4}
	вытекает из свойства~\ref{it:properties-tau-3}
	при $n=0$.
	
	Докажем теперь свойство~\ref{it:properties-tau-5}.
	Имеем
	\begin{multline*}
		I_1=\int_{-1}^1f(x)\T y{}{g,x}\Si{x}^\mu\,dx\\
		=\frac{2^\mu}{\pi(1+y)^\mu}\int_{-1}^1\int_{-1}^1
			f(x)g(R)\Si{R}^{\mu/2}\Si{x}^{\mu/2}\\
			\times\cosmu\frac{dz\,dx}{\sqrt{1-z^2}},
	\end{multline*}
	где~$R$, $\varphi_1$ и~$\varphi$
	определены формулами~\eqref{eq:R-B}.
	Сделав в этом двойном интеграле
	замену переменных по формулам
	\begin{align*}
		x &=Ry+V\sqrt{1-R^2}\sqrt{1-y^2},\\
		z &=-\frac{R\sqrt{1-y^2}-Vy\sqrt{1-R^2}}
			{\sqrt{1-\left(Ry+V\sqrt{1-R^2}\sqrt{1-y^2}\right)^2}},
	\end{align*}
	получим, что
	\begin{displaymath}
		I_1=\int_{-1}^1g(R)\T y{}{f,R}\Si{R}^\mu\,dR,
	\end{displaymath}
	что и требовалось доказать.
	
	Для доказательства свойства~\ref{it:properties-tau-6}
	рассмотрим
	\begin{displaymath}
		I_2=a_n(\T y{}{f,x})
		=\int_{-1}^1\T y{}{f,x}\Px n(x)\Si{x}^\mu\,dx.
	\end{displaymath}
	Используя свойства~\ref{it:properties-tau-5}
	и~\ref{it:properties-tau-3},
	получаем, что
	\begin{multline*}
		I_2=\int_{-1}^1f(x)\T y{}{\Px n,x}\Si{x}^\mu\,dx\\
		=\Py\int_{-1}^1f(x)\Px n(x)\Si{x}^\mu\,dx.
	\end{multline*}
	
	Лемма~\ref{lm:properties-tau} доказана.
\end{proof}

\begin{lmm}\label{lm:inequality}
	Пусть даны числа~$p$ и~$\alpha$ такие,
	что $\allp$;
	\begin{alignat*}2
		\frac12 		&<\alpha\le1
			&\quad &\text{при $p=1$},\\
		1-\frac1{2p}	&<\alpha<\frac32-\frac1{2p}
			&\quad &\text{при $1<p<\infty$},\\
		1 				&\le\alpha<\frac32
			&\quad &\text{при $p=\infty$}.
	\end{alignat*}
	Пусть~$R$ определено формулами~\eqref{eq:R-B}.
	Тогда для любой измеримой на отрезке $[-1,1]$ функции~$f$
	справедливо неравенство
	\begin{displaymath}
		\left\|{\int_{-1}^1\Si R|f(R)|\frac{dz}{\sqrt{1-z^2}}}\right\|
			_{p,\alpha-1}
		\le C\norm f,
	\end{displaymath}
	где постоянная~$C$ не зависит от~$f$ и~$t$.
\end{lmm}

Лемма~\ref{lm:inequality}
дана в работе~\cite{p-berisha:anal-99}.

\begin{lmm}\label{lm:bound-tau}
	Пусть даны числа~$p$, $\mu$ и~$\alpha$ такие,
	что $\allp$,
	$\allmu$;
	\begin{alignat*}2
		-\frac12 	&<\alpha-\frac\mu2\le0
			&\quad &\text{при $p=1$},\\
		-\frac1{2p} &<\alpha-\frac\mu2<\frac12-\frac1{2p}
			&\quad &\text{при $1<p<\infty$},\\
		0			&\le\alpha-\frac\mu2<\frac12
			&\quad
				&\text{при $p=\infty$}.
	\end{alignat*}
	Пусть $f\in\Lp$.
	Тогда справедливо неравенство
	\begin{displaymath}
		\norm{\hatT t{}{f,x}}\le C\frac1{\Co t}\norm f,
	\end{displaymath}
	где постоянная~$C$ не зависит от~$f$, $t$ и~$\mu$.
\end{lmm}

\begin{proof}
	Пусть
	\begin{displaymath}
		I=\norm{\hatT t{}{f,x}}.
	\end{displaymath}
	Тогда
	\begin{multline*}
		I=\frac1{\pi\Co t}\\
			\times\norm{\frac1{\Si{x}^{\mu/2}}
			\int_{-1}^1\Si{R}^{\mu/2}f(R)\cosmu\frac{dz}{\sqrt{1-z^2}}},
	\end{multline*}
	где~$R$, $\varphi_1$ и~$\varphi$
	даны формулами~\eqref{eq:R-B}.
	Применяя лемму~\ref{lm:inequality},
	получаем, что
	\begin{displaymath}
		I\le\Cn\frac1{\Co t}\|\Si{x}^{\mu/2-1}|f(x)|\|_{p,\alpha+1-\mu/2}
		=\lastC\frac1{\Co t}\norm f.
	\end{displaymath}
	
	Лемма~\ref{lm:bound-tau} доказана.
\end{proof}

\begin{lmm}\label{lm:Dtau}
	Пусть функция~$f(x)$
	имеет на каждом отрезке $[a,b]\subset(-1,1)$
	абсолютно непрерывную производную~$f'(x)$,
	$\allmu$
	и пусть $\Dx f(x)\in\Lmu$.
	Тогда
	\begin{enumerate}
	\item\label{it:Dtau-1}
		При фиксированном~$y$
		функция $\T y{}{f,x}$
		имеет на каждом отрезке $[c,d]\subset(-1,1)$
		абсолютно непрерывную производную $\frac d{dx}\T y{}{f,x}$.
	\item\label{it:Dtau-2}
		При фиксированном~$x$
		функция $\T y{}{f,x}$
		имеет на каждом отрезке $[c,d]\subset(-1,1)$
		абсолютно непрерывную производную $\frac d{dy}\T y{}{f,x}$.
	\item\label{it:Dtau-3}
		Для почти всех~$x$ и~$y$
		справедливы равенства
		$$
		\T y{}{\Dx f,x}=\Dx\T y{}{f,x}=\Dy\T y{}{f,x}.
		$$
	\end{enumerate}
\end{lmm}

\begin{proof}
	Рассмотрим функцию
	\begin{displaymath}
		\varphi(x)
		=\frac{\Si{R}^{\mu/2}\cosmu}
			{\Si{x}^{\mu/2}(1+y)^2\sqrt{1-z^2}}f(R),
	\end{displaymath}
	где~$R$, $\varphi_1$ и~$\varphi$
	определены формулами~\eqref{eq:R-B}.
	Нетрудно доказать,
	что при фиксированных~$y$ и~$z$,
	либо $R$ --- монотонная функция от~$x$ на отрезке $[-1,1]$,
	либо, если существует конечно число $\theta_0=\arctg(z\tg t)$,
	$R$ --- монотонная функция от~$x$
	на каждом из отрезков $[-1,\cos\theta_0]$
	и $[\cos\theta_0,1]$.
	Поэтому заключаем,
	что фкнкция~$\varphi'(x)$ --- абсолютно непрерывна
	на каждом отрезке $[c,d]\subset(-1,1)$.
	Отсюда утверждение~\ref{it:Dtau-1}
	следует после применения теоремы Лебега
	о переходе пределом под знаком интеграла.
	
	Учитывая симметричность~$R$ по~$x$ и~$y$,
	аналогичным расуждением
	доказывается абсолютная непрерывность функции
	$\frac d{dy}\T y{}{f,x}$.
	
	Докажем теперь равенство
	\begin{equation}\label{eq:Dxtau}
		\T y{}{\Dx f,x}=\Dx\T y{}{f,x}.
	\end{equation}
	
	Из~\ref{it:Dtau-1} следует,
	что существует $\Dx\T y{}{f,x}$.
	
	Пусть вначале функция~$f(x)$ бесконечно дифференцируема
	и равна нулю вне некоторого отрезка
	$[a,b]\subset(-1,-y)\cup(-y,y)\cup(y,1)$.
	Применяя свойств~\ref{it:properties-tau-5}
	и~\ref{it:properties-tau-3}
	из леммы~\ref{lm:properties-tau},
	получаем
	\begin{multline*}
		I=\int_{-1}^1\T y{}{\Dx f,x}\Px n(x)\Si{x}^\mu\,dx\\
		=\Py\int_{-1}^1\Dx f(x)\Px n(x)\Si{x}^\mu\,dx.
	\end{multline*}
	Интегрируя дважды по частям,
	учитывая, что $f(x)=0$
	и $f'(x)=0$ вне $[a,b]\subset(-1,1)$,
	получим, что
	\begin{displaymath}
		I=\Py\int_{-1}^1\Dx\Px n(x)f(x)\Si{x}^\mu\,dx.
	\end{displaymath}
	Известно, что~\cite[с.171]{erdelyi-m-o-t:transcendental}
	\begin{displaymath}
		\Dx\Px n(x)=-n(n+2\mu+1)\Px n(x).
	\end{displaymath}
	Поэтому
	\begin{displaymath}
		I=-n(n+2\mu+1)\Py\int_{-1}^1f(x)\Px n(x)\Si{x}^\mu\,dx.
	\end{displaymath}
	Отсюда, применяя свойств~\ref{it:properties-tau-3}
	и~\ref{it:properties-tau-5} леммы~\ref{lm:properties-tau}
	и интегрируя дважды по частям,
	учитывая, что $\T y{}{f,x}=0$
	вне некоторого отрезка $[\gamma,\delta]\subset(-1,1)$,
	получаем, что
	\begin{displaymath}
		I=\int_{-1}^1\Dx\T y{}{f,x}\Px n(x)\Si{x}^\mu\,dx.
	\end{displaymath}
	Следовательно,
	при фиксированном~$y$ все коэффициенты Фурье--Якоби
	функции
	\begin{displaymath}
		F(x)=\T y{}{\Dx f,x}-\Dx\T y{}{f,x}
	\end{displaymath}
	по системе многочленов $\left\{\Px n(x)\right\}_{n=0}^\infty$
	равны нулю.
	Из свойства полноты системы
	$\left\{\Px n(x)\right\}_{n=0}^\infty$
	заключаем,
	что $F(x)=0$ почти всюду на $[-1,1]$.
	
	Тем самим равенство~\eqref{eq:Dxtau} доказано
	при фиксированном~$y$
	для функции~$f(x)$ бесконечно дифференцируемой и равной нулю
	вне некоторого отрезка
	$[a,b]\subset(-1,-y)\cup(-y,y)\cup(y,1)$.
	
	Пусть теперь функция $f(x)$
	имеет на каждом отрезке $[a,b]\subset(-1,1)$
	абсолютно непрерывную производную~$f'(x)$.
	Пусть функция~$g(x)$ бесконечно дифференцируема
	и равна нулю
	вне некоторого отрезка $[c,d]\subset(-1,-y)\cup(-y,y)\cup(y,1)$.
	Интегрируя дважды по частям,
	учитывая, что
	\begin{displaymath}
		g(x)\Si{x}^{\mu+1}\frac d{dx}\T y{}{f,x}\to0
	\end{displaymath}
	и
	\begin{displaymath}
		\T y{}{f,x}\Si{x}^{\mu+1}\frac d{dx}g(x)\to0
	\end{displaymath}
	для $x\to-1+0$ и $x\to1-0$,
	получим, что
	\begin{multline*}
		J_1=\int_{-1}^1\Dx\T y{}{f,x}g(x)\Si{x}^\mu\,dx\\
		=\int_{-1}^1\Dx g(x)\T y{}{f,x}\Si{x}^\mu\,dx.
	\end{multline*}
	Применяя лемму~\ref{lm:properties-tau},
	имеем
	\begin{displaymath}
		J_1=\int_{-1}^1f(x)\T y{}{\Dx g,x}\Si{x}^\mu\,dx.
	\end{displaymath}
	
	Пусть теперь
	\begin{displaymath}
		J_2=\int_{-1}^1\T y{}{\Dx f,x}g(x)\Si{x}^\mu\,dx.
	\end{displaymath}
	Аналогично применяя лемму~\ref{lm:properties-tau}
	и интегрируя дважды по частям,
	получаем, что
	\begin{displaymath}
		J_2=\int_{-1}^1\Dx\T y{}{g,x}f(x)\Si{x}^\mu\,dx.
	\end{displaymath}
	Следовательно
	\begin{displaymath}
		J_2-J_1=\int_{-1}^1\left(\Dx\T y{}{g,x}-\T y{}{\Dx g,x}\right)
		f(x)\Si{x}^\mu\,dx.
	\end{displaymath}
	Но для бесконечно дифференцируемой
	и равной нулю
	вне некоторого отрезка
	$[c,d]\subset(-1,-y)\cup(-y,y)\cup(y,1)$
	функции~$g(x)$
	доказано равенство~\eqref{eq:Dxtau}
	при фиксированном~$y$
	и почти всех $x\in[-1,1]$.
	Отсюда
	\begin{displaymath}
		J_2-J_1=\int_{-1}^1\left(\T y{}{\Dx f,x}-\Dx\T y{}{f,x}\right)
		g(x)\Si{x}^\mu\,dx=0
	\end{displaymath}
	для всех~$y$.
	Так как,
	отрезок
	$[c,d]\subset(-1,-y)\cup(-y,y)\cup(y,1)$
	выбран любым,
	а функция~$g(x)$ --- любая бесконечно дифференцируемая функция,
	равная нулю вне отрезка $[c,d]$,
	то равенство~\eqref{eq:Dxtau} справедливо
	почти всюду на $[-1,1]$ при фиксированном~$y$.
	
	Равенство
	\begin{displaymath}
		\T y{}{\Dx f,x}=\Dy\T y{}{f,x}
	\end{displaymath}
	доказывается аналогично.
	
	Лемма~\ref{lm:Dtau} доказана.
\end{proof}

\begin{lmm}\label{lm:tauuDx}
	Пусть функция~$f(x)$
	имеет на каждом отрезке
	$[a,b]\subset(-1,1)$
	абсолютно непрерывную производную~$f'(x)$.
	Тогда для почти всех $x\in[-1,1]$
	выполнены следующие равенства
	\begin{multline*}
		\T y{}{f,x}-f(x)=\int_1^y(1-v)^{-1}(1+v)^{-2\mu-1}\\
			\times\int_1^v(1+u)^{2\mu}\T u{}{\Dx f,x}\,du\,dv
	\end{multline*}
	и
	\begin{multline*}
		\T y{}{f,x}-\T 0{}{f,x}=-\int_0^y(1-v)^{-1}(1+v)^{-2\mu-1}\\
			\times\int_v^{-1}(1+u)^{2\mu}\T u{}{\Dx f,x}\,du\,dv.
	\end{multline*}
\end{lmm}

\begin{proof}
	Докажем первое равенство леммы.
	Если функция~$f(x)$ --- бесконечно дифференцируема
	и равна нулю вне некоторого отрезка
	$[a,b]\subset(-1,-y)\cup(-y,y)\cup(y,1)$,
	то, применяя леммы~\ref{lm:Dtau} и~\ref{lm:properties-tau},
	получаем, что
	\begin{multline*}
		\int_1^y(1-v)^{-1}(1+v)^{-2\mu-1}
			\int_1^v(1+u)^{2\mu}\T u{}{\Dx f,x}\,du\,dv\\
		=\int_1^y(1-v)^{-1}(1+v)^{-2\mu-1}
			\int_1^v(1+u)^{2\mu}D_{u,0,2\mu}\T u{}{f,x}\,du\,dv\\
		=\T y{}{f,x}-f(x)
	\end{multline*}
	при почти всех $x\in[-1,1]$.
	
	Пусть теперь~$f(x)$
	имеет на каждом отрезке $[a,b]\subset(-1,1)$
	абсолютно непрерывную производную~$f'(x)$.
	Пусть функция~$g(x)$ бесконечно дифференцируема
	и равна нулю вне некоторого отрезка
	$[c,d]\subset(-1,-y)\cup(-y,y)\cup(y,1)$.
	Тогда меняя порядок интегрирования,
	потом применяя лемму~\ref{lm:properties-tau}
	и, рассуждая аналогично
	как при доказательстве леммы~\ref{lm:Dtau},
	интегрируя дважды по частям,
	нетрудно доказать, что
	\begin{multline*}
		J=\int_{-1}^1\int_1^y(1-v)^{-1}(1+v)^{-2\mu-1}\\
			\times\int_1^v(1+u)^{2\mu}
			\T u{}{\Dx f,x}g(x)\Si{x}^\mu\,du\,dv\,dx\\
		=\int_{-1}^1f(x)\Si{x}^\mu\int_1^y(1-v)^{-1}(1+v)^{-2\mu-1}\\
			\times\int_1^v(1+u)^{2\mu}\Dx\T u{}{g,x}\,du\,dv\,dx.
	\end{multline*}
	Но для бесконечно дифференцируемой
	и равной нулю вне некоторого отрезка
	$[c,d]\subset(-1,-y)\cup(-y,y)\cup(y,1)$
	функции~$g(x)$
	уже доказано первое равенство леммы при почти всех $x\in[-1,1]$.
	Следовательно,
	\begin{displaymath}
		J=\int_{-1}^1\left(\T y{}{f,x}-f(x)\right)g(x)\Si{x}^\mu\,dx.
	\end{displaymath}
	Отсюда,
	первое равенство леммы
	вытекает в силу произвольности отрезка $[c,d]$
	и функции~$g(x)$.
	
	Второе равенство леммы доказывается аналогично.
\end{proof}

\begin{cor}
	Пусть функция~$f(x)$ имеет на каждом отрезке $[a,b]\subset(-1,1)$
	абсолютно непрерывную производную~$f'(x)$.
	Тогда для почти всех $x\in[-1,1]$
	выполнены следующие равенства
	\begin{multline*}
		\hatT t{}{f,x}-f(x)
		=\int_0^t\sincosv\\
			\times\int_0^v\hatT u{}{\Dx f,x}\sincosu\,du\,dv
	\end{multline*}
	и
	\begin{multline*}
		\hatT t{}{f,x}-\hatT{\pi/2}{}{f,x}
		=-\int_{\pi/2}^t\sincosv\\
			\times\int_v^\pi\hatT u{}{\Dx f,x}\sincosu\,du\,dv.
	\end{multline*}
\end{cor}

Первое равенство
следует сразу из первого равенства леммы~\ref{lm:tauuDx},
подстановкой $\cos u$ и $\cos v$
вместо~$u$ и~$v$ соответственно.
Аналогично этому,
второе равенство следует из второго равенства леммы~\ref{lm:tauuDx}.

\begin{lmm}\label{lm:bernshtein-markov}
	Пусть~$P_n(x)$ --- алгебраический многочлен степени не выше,
	чем $n-1$,
	$\allp$, $\rho\ge0$;
	\begin{alignat*}2
		\alpha &>-\frac1p	&\quad &\text{при $1\le p<\infty$},\\
		\alpha &\ge0 	 	&\quad &\text{при $p=\infty$}.
	\end{alignat*}
	Тогда справедливы неравенства
	\begin{gather*}
		\|P'_n(x)\|_{p,\alpha+1/2}
		\le\Cn n\norm{P_n},\\
			\norm{P_n}
		\le\Cn n^{2\rho}\|P_n\|_{p,\alpha+\rho},
	\end{gather*}
	где постоянные~$\prevC$ и~$\lastC$
	не зависят от~$n$ $(n\in\numN)$.
\end{lmm}

Лемма~\ref{lm:bernshtein-markov}
доказана в работе~\cite{potapov:vestnik-60a}.

\begin{lmm}\label{lm:E-D}
	Пусть даны числа~$p$, $\alpha$ и~$\mu$ такие,
	что $\allp$,
	$\allmu$;
	\begin{alignat*}2
		-\frac12 	&<\alpha\le\mu
			&\quad &\text{при $p=1$},\\
		-\frac1{2p}	&<\alpha<\mu+\frac12-\frac1{2p}
			&\quad &\text{при $1<p<\infty$},\\
		0			&\le\alpha<\mu+\frac12
			&\quad &\text{при $p=\infty$}.
	\end{alignat*}
	Пусть $f\in\AD$.
	Тогда справедливо неравенство
	\begin{displaymath}
		\E\le C\frac1{n^2}\norm{\Dx f(x)},
	\end{displaymath}
	где постоянная~$C$ не зависит от~$f$, $n$ и~$\mu$.
\end{lmm}

\begin{proof}
	Для фиксированного натурального числа $q>\mu$
	выберем натуральное число~$n$ такое, что
	\begin{displaymath}
		\frac{n-1}{q+2}<m<\frac{n-1}{q+2}+1.
	\end{displaymath}
	
	Нетрудно доказать,
	что при условиях леммы имеем $f\in\Lmu$.
	В работе~\cite{potapov:trudy-74} доказано,
	что функция
	\begin{displaymath}
		Q(x)=\frac1{\gamma_m}\int_0^\pi T_{\cos t}(f,x,\mu)\krn\,dt,
	\end{displaymath}
	где
	\begin{displaymath}
		\gamma_m=\int_0^\pi\krn\,dt,
	\end{displaymath}
	есть алгебраический многочлен степени не выше, чем $n-1$.
	Поэтому, применяя обобщенное неравенство Минковского,
	имеем
	\begin{multline*}
		\E\le\norm{f-Q}\\
		\le\frac1{\gamma_m}\int_0^\pi
		\norm{T_{\cos t}(f,x,\mu)-f(x)}\krn\,dt.
	\end{multline*}
	В работе~\cite[с.47]{potapov:vestnik-83}
	доказано, что для всех~$t$ справедливо неравенство
	\begin{displaymath}
		\norm{T_{\cos t}(f,x,\mu)-f(x)}\le\Cn t^2\norm{\Dx f(x)}.
	\end{displaymath}
	Поэтому
	\begin{displaymath}
		\E
		\le\lastC\norm{\Dx f(x)}\frac1{\gamma_m}
			\int_0^\pi t^2\krn\,dt.
	\end{displaymath}
	Применяя стандартную оценку ядра Джексона,
	получаем, что
	\begin{displaymath}
		\E\le\Cn\frac1{m^2}\norm{\Dx f(x)}
		\le\Cn\frac1{n^2}\norm{\Dx f(x)}.
	\end{displaymath}
	
	Лемма~\ref{lm:E-D} доказана.
\end{proof}

\section{Основные утверждения}

\begin{thm}\label{th:w-K}
	Пусть даны числа~$p$, $\alpha$ и~$\mu$ такие,
	что $\allp$,
	$\allmu$;
	\begin{alignat*}2
		-\frac12 	&<\alpha-\frac{\mu}2\le0
			&\quad &\text{при $p=1$},\\
		-\frac1{2p} &<\alpha-\frac{\mu}2<\frac12-\frac1{2p}
			&\quad &\text{при $1<p<\infty$},\\
		0 			&\le\alpha-\frac{\mu}2<\frac12
			&\quad &\text{при $p=\infty$}.
	\end{alignat*}
	Пусть $f\in\Lp$.
	Тогда для всех $\delta\in(0,\pi)$
	имеют место неравенства
	\begin{displaymath}
		\Cn\K\le\w\le\Cn\frac1{\Co\delta}\K,
	\end{displaymath}
	где положительные постоянные~$\prevC$ и~$\lastC$
	не зависят от~$f$, $\delta$ и~$\mu$.
\end{thm}

\begin{proof}
	Покажем,
	что для любой функции $g(x)\in\AD$
	и любого $t\in(-\pi,\pi)$
	справедливо неравенство
	\begin{equation}\label{eq:Dl-Dx}
		\norm{\hatT t{}{g,x}-g(x)}
		\le\Cn\frac1{\Co t}t^2\norm{\Dx g(x)},
	\end{equation}
	где постоянная~$\lastC$ не зависит от~$g$, $t$ и~$\mu$.
	
	Пусть $0<t\le\frac\pi2$.
	Тогда по следствию из леммы~\ref{lm:tauuDx},
	применяя обобщенное неравенство Минковского
	и лемму~\ref{lm:bound-tau},
	получим
	\begin{multline*}
		I_1=\norm{\hatT t{}{g,x}-g(x)}
		\le\int_0^t\sincosv\\
			\times\int_0^v\norm{\hatT u{}{\Dx g,x}}\sincosu\,du\,dv\\
		\le\Cn\norm{\Dx g(x)}\int_0^t\sincosv\\
			\times\int_0^v\varsincosu\,du\,dv.
	\end{multline*}
	Отсюда, учитывая,
	что при $0<t\le\frac\pi2$
	имеем
	\begin{displaymath}
		\int_0^t\sincosv\int_0^v\varsincosu\,du\,dv\le\Cn t^2,
	\end{displaymath}
	получаем, что
	\begin{displaymath}
		I_1\le\Cn t^2\norm{\Dx g(x)}
		\le\lastC\frac1{\Co t}t^2\norm{\Dx g(x)}.
	\end{displaymath}
	
	Пусть $\frac\pi2\le t<\pi$.
	Тогда по следствию из леммы~\ref{lm:tauuDx},
	применяя обобщенное неравенство Минковского,
	потом лемму~\ref{lm:bound-tau},
	получаем, что
	\begin{multline*}
		I_2=\norm{\hatT t{}{g,x}-\hatT {\pi/2}{}{g,x}}\\
		\le\Cn\norm{\Dx g(x)}\int_{\pi/2}^t\sincosv\\
			\times\int_v^\pi\varsincosu\,du\,dv.
	\end{multline*}
	Учитывая, что для $\frac\pi2\le t<\pi$
	\begin{multline*}
		\int_{\pi/2}^t\sincosv\int_v^\pi\varsincosu\,du\,dv\\
		\le\Cn\frac1{\Co t},
	\end{multline*}
	получаем, что
	\begin{equation}\label{eq:I2-Dx}
		I_2\le\Cn\frac1{\Co t}\norm{\Dx g(x)}
		\le\lastC\frac1{\Co t}t^2\norm{\Dx g(x)}.
	\end{equation}
	Поскольку
	\begin{multline*}
		\norm{\hatT t{}{g,x}-g(x)}\\
		\le\norm{\hatT t{}{g,x}-\hatT {\pi/2}{}{g,x}}
			+\norm{\hatT {\pi/2}{}{g,x}-g(x)},
	\end{multline*}
	то, применяя неравенство~\eqref{eq:I2-Dx}
	и уже доказанные
	для $0<t\le\frac\pi2$
	неравенства~\eqref{eq:Dl-Dx},
	получаем, что для $\frac\pi2\le t<\pi$
	\begin{displaymath}
		\norm{\hatT t{}{g,x}-g(x)}
		\le\Cn\frac1{\Co t}t^2\norm{\Dx g(x)}.
	\end{displaymath}
	
	Таким образом,
	неравенство~\eqref{eq:Dl-Dx}
	доказано при $0<t<\pi$.
	
	Так как, $\T{\cos t}{}{g,x}=\T{\cos(-t)}{}{g,x}$,
	то можно считать,
	что неравенство~\eqref{eq:Dl-Dx}
	справедливо для $0<|t|<\pi$.
	
	Пусть теперь $f\in\Lp$ и $0<|t|\le\delta<\pi$.
	Тогда для любой функции $g(x)\in\AD$,
	применяя лемму~\ref{lm:properties-tau},
	имеем
	\begin{multline*}
		\norm{\hatT t{}{f,x}-f(x)}\\
		\le\norm{\hatT t{}{f-g,x}}
			+\norm{\hatT t{}{g,x}-g(x)}+\norm{g-f}.
	\end{multline*}
	Применяя лемму~\ref{lm:bound-tau}
	и неравенство~\eqref{eq:Dl-Dx},
	получаем, что
	\begin{displaymath}
		\norm{\hatT t{}{f,x}-f(x)}
		\le\Cn\frac1{\Co t}
			\left(\norm{f-g}+t^2\norm{\Dx g(x)}\right),
	\end{displaymath}
	где постоянная~$\lastC$ не зависит от~$f$, $g$, $t$ и~$\mu$.
	Следовательно,
	переходя к точной нижней грани в этом неравенстве
	при $|t|\le\delta$ и $g(x)\in\AD$,
	получаем правое неравенство теоремы
	для $0<\delta<\pi$.
	
	Для доказательства левого неравенства
	рассмотрим функцию
	\begin{multline*}
		\gd=\frac1{\kap}\int_0^\delta\sincosv\\
			\times\int_0^v\hatT u{}{f,x}\sincosu\,du\,dv,
	\end{multline*}
	где
	\begin{displaymath}
		\kap=\int_0^\delta\sincosv\int_0^v\sincosu\,du\,dv.
	\end{displaymath}
	
	Применяя обобщенное неравенство Минковского
	и лемму~\ref{lm:bound-tau},
	получаем, что
	\begin{displaymath}
		\norm{\gd}\le\Cn\frac1{\Co\delta}\norm f,
	\end{displaymath}
	т.е. $\gd\in\Lp$.
	
	Пусть $0<\delta\le\frac\pi2$.
	Тогда легко показать, что
	\begin{equation}\label{eq:kap-delta}
		\kap\ge\Cn\delta^2,
	\end{equation}
	где постоянняая~$\lastC$ не зависит от~$\delta$ и~$\mu$.
	
	Нетрудно показать, что при условиях теоремы,
	из $f\in\Lp$
	следует $f\in\Lmu$.
	
	Обозначим
	\begin{displaymath}
		g(x)=-\int_0^x\Si{y}^{-\mu-1}\int_y^1\Si{z}^\mu
			\left(f(z)-\frac{c_1}{c_0}\right)\,dz\,dy,
	\end{displaymath}
	где $c_1=\int_{-1}^1f(z)\Si{z}^\mu\,dz$,
	$c_0=\int_{-1}^1\Si{z}^\mu\,dz$.
	Ясно, что $g(x)\in\AD$.
	
	Поскольку
	\begin{displaymath}
		\Dx g(x)=f(x)-\frac{c_1}{c_0},
	\end{displaymath}
	то
	\begin{multline*}
		\gd=\frac1{\kap}\int_0^\delta\sincosv\\
			\times\int_0^v\hatT u{}{\Dx g,x}
			\sincosu\,du\,dv+\frac{c_1}{c_0}.
	\end{multline*}
	Применяя следствие из леммы~\ref{lm:tauuDx},
	получаем, что
	\begin{displaymath}
		\gd=\frac1{\kap}\left(\hatT\delta{}{g,x}-g(x)\right)
			+\frac{c_1}{c_0}.
	\end{displaymath}
	Применяя к этому равенству
	оператор $\Dx$ и лемму~\ref{lm:Dtau},
	находим
	\begin{multline*}
		\Dx\gd=\frac1{\kap}\left(\hatT\delta{}{\Dx g,x}-\Dx g(x)\right)\\
		=\frac1{\kap}\left(\hatT\delta{}{f,x}-f(x)\right).
	\end{multline*}
	Следовательно,
	применяя лемы~\ref{lm:bound-tau} и~\ref{lm:Dtau},
	заключаем, что $\gd\in\AD$.
	
	Из последнего равенства и неравенства~\eqref{eq:kap-delta},
	получаем, что
	\begin{displaymath}
		\norm{\Dx\gd}
		\le\Cn\frac1{\delta^2}\norm{\hatT\delta{}{f,x}-f(x)},
	\end{displaymath}
	откуда следует, что
	\begin{displaymath}
		\norm{\Dx\gd}\le\lastC\frac1{\delta^2}\w.
	\end{displaymath}
	
	С другой страны,
	применяя обобщенное неравенство Минковского,
	получаем, что
	\begin{multline*}
		\norm{f(x)-\gd}
		\le\frac1{\kap}\int_0^\delta\sincosv\\
			\times\int_0^v\norm{f(x)-\hatT u{}{f,x}}\sincosu\,du\,dv
		\le\w.
	\end{multline*}
	
	Таким образом, для $0<\delta\le\frac\pi2$
	показано, что
	\begin{displaymath}
		I_\delta=\norm{f(x)-\gd}+\delta^2\norm{\Dx\gd}\le\Cn\w.
	\end{displaymath}
	
	Но для $\frac\pi2\le\delta<\pi$
	имеем
	\begin{multline*}
		I_\delta\le\pi^2\bigl(\norm{f(x)-\gd}+\norm{\Dx\gd}\bigr)\\
		\le\Cn\wpar{f,1}\le\lastC\w,
	\end{multline*}
	и следовательно,
	левое неравенство теоремы справедливо при всех
	$0<\delta<\pi$.
	
	Для $\delta=0$ утверждение теоремы тривиально.
	
	Теорема~\ref{th:w-K} полностью доказана.
\end{proof}

\begin{proof}[Доказательство теоремы~\ref{th:jackson}]
	\setcounter{const}2
	Для любой функции $g(x)\in\AD$,
	применяя лемму~\ref{lm:E-D},
	имеем
	\begin{displaymath}
		\E\le\Epar n{f-g}+\Epar ng
		\le\norm{f-g}+\Cn\frac1{n^2}\norm{\Dx g(x)},
	\end{displaymath}
	где постоянная~$\lastC$ не зависит от~$g$ и~$n$.
	Отсюда,
	переходя к точной нижней грани по всем $g(x)\in\AD$
	и потом применяя теорему~\ref{th:w-K},
	получим
	\begin{displaymath}
		\E\le\Cn\Kpar{f,\frac1n}\le\Cn\wpar{f,\frac1n}.
	\end{displaymath}
	
	Левое неравенство теоремы доказано.
	
	Докажем правое неравенство теоремы.
	Пусть~$P_n(x)$
	--- алгебраический многочлен наилучшего приближения
	для~$f$,
	степени не выше, чем $n-1$.
	Пусть~$k$ выбрано так,
	что
	\begin{equation}\label{eq:k}
		2^k\le n<2^{k+1}.
	\end{equation}
	Из теоремы~\ref{th:w-K},
	учитывая, что $P_{2^k}(x)\in\AD$,
	следует, что
	\begin{displaymath}
		\wpar{f,\frac1n}\le\Cn\left(\norm{f-P_{2^k}}
			+\frac1{n^2}\norm{\Dx P_{2^k}(x)}\right).
	\end{displaymath}
	
	Так как,
	\begin{displaymath}
		\Dx P_{2^k}(x)
		=\sum_{\nu=0}^{k-1}\Dx
			\left(P_{2^{\nu+1}}(x)-P_{2^\nu}(x)\right),
	\end{displaymath}
	учитывая,
	что из леммы~\ref{lm:bernshtein-markov}
	следует, что
	\begin{multline*}
		\norm{\Dx P_n(x)}\le\norm{\Si{x}P''_n(x)}
			+(2\mu+2)\norm{P'_n(x)}\\
		\le\Cn n^2\norm{P_n},
	\end{multline*}
	получаем
	\begin{multline*}
		\norm{\Dx P_{2^k}(x)}
		\le\Cn\sum_{\nu=0}^{k-1}2^{2(\nu+1)}
			\norm{P_{2^{\nu+1}}-P_{2^\nu}}\\
		\le2\lastC\sum_{\nu=0}^{k-1}2^{2(\nu+1)}\Epar{2^\nu}f.
	\end{multline*}
	Поэтому, учитывая неравенство~\eqref{eq:k},
	имеем
	\begin{displaymath}
		\wpar{f,\frac1n}
		\le\Cn\frac1{n^2}
			\sum_{\nu=0}^k 2^{2(\nu+1)}\Epar{2^\nu}f.
	\end{displaymath}
	
	Замечая, что для $\nu=1,\dots,k$
	\begin{displaymath}
		\sum_{j=2^{\nu-1}}^{2^\nu-1}j\Epar j f
		\ge2^{2(\nu-1)}\Epar{2^\nu}f,
	\end{displaymath}
	находим
	\begin{multline*}
		\wpar{f,\frac1n}
		\le\Cn\frac1{n^2}\biggl(4\Epar1f+\sum_{\nu=1}^k
			\sum_{j=2^{\nu-1}}^{2^\nu-1}j\Epar j f\biggr)\\
		\le\Cn\frac1{n^2}\sum_{\nu=1}^n\nu\Epar\nu f.
	\end{multline*}
	
	Теорема~\ref{th:jackson} доказана.
\end{proof}

\end{document}